\documentclass{amsproc}
\usepackage{amsfonts}

\theoremstyle{plain}

\newtheorem{corollary}{Corollary}

\newtheorem{definition}{Definition}
\newtheorem{example}{Example}

\newtheorem{lemma}{Lemma}

\newtheorem{problem}{Problem}
\newtheorem{proposition}{Proposition}
\newtheorem{remark}{Remark}

\newtheorem{theorem}{Theorem}
\numberwithin{equation}{section}

\begin{document}
\title[Invariance and means]{Invariance in a class of operations related to
weighted quasi-geometric means}
\author{Jimmy Devillet}
\curraddr{Mathematics Research Unit, University of Luxembourg, Maison du
Nombre, 6, avenue de la Fonte, L-4364 Esch-sur-Alzette, Luxembourg}
\email{jimmy.devillet@uni.lu}
\author{Janusz Matkowski}
\curraddr{Faculty of Mathematics Computer Science and Econometrics,
Univerity of Zielona G\'{o}ra, Szafrana 4A, PL 65-516 Zielona G\'{o}ra,
Poland}
\email{J.Matkowski@wmie.uz.zgora.pl}

\begin{abstract}
Let $I\subset (0,\infty )$ be an interval that is closed with respect to the
multiplication. The operations $C_{f,g}\colon I^{2}\rightarrow I$ of the
form
\begin{equation*}
C_{f,g}\left( x,y\right) =\left( f\circ g\right) ^{-1}\left( f\left(
x\right) \cdot g\left( y\right) \right) \text{,}
\end{equation*}%
where $f,g$ are bijections of $I$ are considered. Their connections with
generalized weighted quasi-geometric means is presented. It is shown that invariance\
question within the class of this operations leads to means of iterative
type and to a problem on a composite functional equation. An application of
the invariance identity to determine effectively the limit of the sequence
of iterates of some generalized quasi-geometric mean-type mapping, and the
form of all continuous functions which are invariant with respect to this
mapping are given. The equality of two considered operations is also discussed.
\end{abstract}

\maketitle

\section{Introduction}

\footnotetext{\textit{2010 Mathematics Subject Classification. }Primary:
26A18, 26E60, 39B12.
\par
\textit{Keywords and phrases:} invariant functions, mean, invariant mean,
reflexivity, iteration, functional equation
\par
{}}

Let $X$ be a set and let $\oplus \colon X^{2}\rightarrow X$ be a bisymmetric
operation. In \cite{Matko18}, given bijective functions $f,g\colon
X\rightarrow X$, the class of operations $D_{f,g}\colon X^{2}\rightarrow X$
defined by
\begin{equation*}
D_{f,g}(x,y)=(f\circ g)^{-1}(f(x)\oplus g(y)),\qquad x,y\in X,
\end{equation*}%
where "$\circ $" stands for the composition, was introduced and their
relation with means was investigated in the case where $X\subset \mathbb{R}$
is an open real interval and $\oplus $ is the usual addition. In particular,
given arbitrary bijective functions $f,g,h\colon X\rightarrow X$ such that $%
D_{f,g}$ and $D_{g,h}$ are reflexive, it was shown (see \cite[Theorem 1]%
{Matko18}) that the function $D_{f\circ g,g\circ h}$ is invariant with
respect to the mapping $\left( D_{f,g},D_{g,h}\right) :X^{2}\rightarrow
X^{2},$ i.e. that
\begin{equation}
D_{f\circ g,g\circ h}\circ \left( D_{f,g},D_{g,h}\right) =D_{f\circ g,g\circ
h}\text{.}  \tag{1}
\end{equation}

In the present paper, given an interval $I\subset (0,\infty )$ that is
closed with respect to the multiplication and bijective functions $f,g\colon I\rightarrow
I $, we consider the class of operations $C_{f,g}:I^{2}\rightarrow I$ of the
form%
\begin{equation*}
C_{f,g}\left( x,y\right) =\left( f\circ g\right) ^{-1}\left( f\left(
x\right) \cdot g\left( y\right) \right) \text{,}
\end{equation*}%
where "$\cdot $" (which will be sometimes omitted) stands for the usual
multiplication. (Of course $C_{f,g}$ form a subclass of the operations $%
D_{f,g}.)$

The similarity of $C_{f,g}$ to the generalized weighted quasi-geometric
means introduced in Section 2, motivates our considerations in Section 3,
where we examine conditions under which $C_{f,g}$ is a bivariable mean. The
reflexivity property of every mean leads to the iterative functional
equation
\begin{equation*}
f\left( g\left( x\right) \right) =f\left( x\right) \cdot g\left( x\right)
\text{, \ \ \ }x\in I,
\end{equation*}%
where $I=\left( 1,\infty \right) $ (or $I=\left[ 1,\infty \right) $) and the
functions $f$ and $g$ are unknown. Under some natural conditions, Theorem 3,
the main result of this section, says in particular that $C_{f,g}$ is a
bivariable mean in $\left( 1,\infty \right) $, if and only if,
\begin{equation*}
f=\prod_{k=0}^{\infty }g^{-k},
\end{equation*}%
where the infinite product $\prod_{k=0}^{\infty }g^{-k}$ of iterates of $%
g^{-1}\,\ $converges uniformly on compact subsets of $\left( 1,\infty
\right) $; moreover $C_{f,g}=\mathcal{C}_{g}$, where
\begin{equation*}
\mathcal{C}_{g}\left( x,y\right) :=\left( \prod_{k=0}^{\infty
}g^{-k+1}\right) ^{-1}\left( \prod_{k=0}^{\infty }g^{-k}\left( x\right)
\cdot g\left( y\right) \right) ,\ \ \ \ x,y>1,
\end{equation*}%
which justify the names: \textit{iterative type mean} for $\mathcal{C}_{g}$,
and \textit{iterative generator} of this mean for $g$.

If $C_{f,g}=\mathcal{C}_{g}$ and $C_{g,h}=\mathcal{C}_{h},$ then, in view of
\cite[Theorem 1]{Matko18}, the function $C_{f\circ g,g\circ h}$\ is
invariant with respect to the mean type mapping $\left( \mathcal{C}_{g},%
\mathcal{C}_{h}\right) $.

On the other hand $\mathcal{C}_{g}$ (and $\mathcal{C}_{h}$) is a very
special case of \textit{generalized weighted quasi-geometric mean} $%
G_{\varphi ,\psi }:I^{2}\rightarrow I$, of the form%
\begin{equation*}
G_{\varphi ,\psi }\left( x,y\right) =\left( \varphi \cdot \psi \right)
^{-1}\left( \varphi \left( x\right) \cdot \psi \left( y\right) \right) \text{%
, \ \ \ \ }x,y\in I\text{,}
\end{equation*}%
where $\varphi ,\psi :I\rightarrow (0,\infty )$ are continuous, of same type
monotonicity, and such that $\varphi \cdot \psi $ is strictly monotonic in
the interval $I$ (see Section 2). Moreover, since generalized weighted
quasi-geometric means can be considered as generalized strict weighted
quasi-arithmetic means (see Section 2), if $G_{\varphi ,\psi }$ and $G_{\psi
,\gamma }$ are two generalized weighted quasi-geometric means, then $%
G_{\varphi \cdot \psi ,\psi \cdot \gamma },$ the mean of the same type, is a
unique mean that is invariant with respect to the mean-type mapping $\left(
G_{\varphi ,\psi },G_{\psi ,\gamma }\right) :I^{2}\rightarrow I^{2}$ (see
\cite{Matko14}). \ This leads to natural and equivalent questions: \ is $%
C_{f\circ g,g\circ h}$ a mean; is the operation $C_{f\circ g,g\circ h}$ a
generalized weighted quasi-geometric mean; or simply can the three
operations $C_{f,g}$, $C_{g,h}$ and $C_{f\circ g,g\circ h}$ be means
simultaneously? In Section 4 we prove that it can happen iff
\begin{equation*}
g=\prod_{i=0}^{\infty }h^{-i}\text{, \ \ \ \ \ \ \ \ \ }f=\prod_{j=0}^{%
\infty }\left( \prod_{i=0}^{\infty }h^{-i}\right) ^{-j}\text{, }
\end{equation*}%
and $h$ satisfies "strongly" composite functional equation
\begin{equation*}
\text{\ }\left( \prod_{j=0}^{\infty }\left( \prod_{i=0}^{\infty
}h^{-i}\right) ^{-j+1}\right) =\prod_{k=0}^{\infty }\left(
\prod_{i=0}^{\infty }h^{-i+1}\right) ^{-k},
\end{equation*}%
where $i,j,k$ stand for the indices of iterates of the suitable functions
(Theorem 4). We propose as an open problem to find the continuous solutions $%
h:\left( 1,\infty \right) \rightarrow \left( 1,\infty \right) .$ In
illustrative Example 2 (with $h\left( x\right) =x^{\frac{1}{w}}$ and $w\in
\left( 0,1\right) )$ we get $\mathcal{C}_{g}\left( x,y\right) =x^{1-w}y^{w},$
$\mathcal{C}_{h}\left( x,y\right) =x^{w}y^{1-w},$ and $C_{f\circ g,g\circ h}=%
\mathcal{G}^{2w(1-w)}$, where $\mathcal{G}\left( x,y\right) :=\sqrt{xy}$ is
the geometric mean, which shows $C_{f\circ g,g\circ h}$ need not be a mean.

In section 5 we apply the invariance identity to determine effectively the
limit of the sequence of iterates of some generalized weighted
quasi-geometric mean-type mappings, as well as, to find the form of all
continuous functions which are invariant with respect to these mappings.

In section 6 the equality of two considered operations is discussed.

\section{Generalized weighted quasi-geometric means}

Let $I\subset \mathbb{R}$ be an interval. Recall that a function $M\colon
I^{2}\rightarrow \mathbb{R}$ is said to be a \emph{(bivariable) mean} if it
is internal, i.e.
\begin{equation*}
\min \left( x,y\right) \leq M\left( x,y\right) \leq \max \left( x,y\right)
\text{, \ \ \ \ \ }x,y\in I,
\end{equation*}%
and a \emph{(bivariable) strict mean} if it is a mean and these inequalities
are sharp for all $x\neq y$.

Recall some well known properties of means.

\begin{remark}
The following three conditions are equivalent:

\ (i) a function $M\colon I^{2}\rightarrow \mathbb{R}$ is a mean in an
interval $I;$

(ii) $M\left( J^{2}\right) \subset J$ for every subinterval $J\subset I$;

(iii) \ $M\left( J^{2}\right) =J$ for every subinterval $J\subset I.$
\end{remark}

\begin{remark}
(i) If $M$ is a mean in $I,$ then $M$ is \emph{reflexive, }i.e.%
\begin{equation*}
M\left( x,x\right) =x\text{, \ \ \ \ \ }x\in I\text{,}
\end{equation*}%
(but the converse implication fails).

(ii) If a function $M\colon I^{2}\rightarrow \mathbb{R}$ is reflexive and
(strictly) increasing in each variable then it is a (strict) mean in
$I$.
\end{remark}

Among many classes of means, one of the most popular is the family of
quasi-arithmetic means, which have the following generalization.

If the functions $\varphi ,\psi :I\rightarrow \mathbb{R}$ are continuous and
both strictly increasing or both strictly decreasing then the function $%
A_{\varphi ,\psi }:I^{2}\rightarrow I$ defined by
\begin{equation*}
A_{\varphi ,\psi }(x,y):=(\varphi +\psi )^{-1}(\varphi (x)+\psi (y)),
\end{equation*}%
is a strict mean, called a \textit{generalized strict weighted
quasi-arithmetic mean in }$I$ (see \cite{Matko10}), where (with a little
modified notation)\ the following result is proved:

\begin{theorem}
\label{thm:eq} Suppose that $A_{\varphi ,\psi }$ and $A_{\Phi ,\Psi }$ are
\textit{generalized strict weighted quasi-arithmetic means in }$I.$\textit{\
Then }$A_{\Phi ,\Psi }=A_{\varphi ,\psi }$ if and only if there exist $%
\alpha ,\beta ,\gamma \in \mathbb{R}$, $\alpha \neq 0,$ such that%
\begin{equation*}
\Phi (x)=\alpha \varphi (x)+\beta ,\qquad \Psi (x)=\alpha \psi (x)+\gamma
,\qquad x\in I.
\end{equation*}
\end{theorem}

If the functions $f,g:I\rightarrow \left( 0,\infty \right) $ are continuous
and both strictly increasing or both strictly decreasing then, similarly as
in the case $A_{\varphi ,\psi }$ one can show that the function $%
G_{f,g}:I^{2}\rightarrow I$ defined by
\begin{equation*}
G_{f,g}(x,y):=(f\cdot g)^{-1}(f(x)g(y)),
\end{equation*}%
is a strict mean (Here "$\cdot $" stands for the usual multiplication of
functions: $\left( f\cdot g\right) \left( x\right) :=f\left( x\right)
g\left( x\right) $ for every $x\in I,$ and we omit writing it later). In the
sequel, $G_{f,g}$ is called a \textit{generalized weighted quasi-geometric
mean in }$I$ of generators $f$ and $g.$

The following remark, which is easy to verify, provides a one-to-one
correspondence between the class of generalized strict weighted
quasi-arithmetic means and the class of generalized weighted quasi-geometric
means.

\begin{remark}
\label{rem:bij} For all continuous and of the same type strict monotonicity
functions $f,g\colon I\rightarrow (0,\infty ),$%
\begin{equation*}
G_{f,g}=A_{\log \circ f\text{,}\log \circ g};
\end{equation*}%
and, for all continuous and of the same type strict monotonicty functions $%
\varphi ,\psi \colon I\rightarrow \mathbb{R}$,%
\begin{equation*}
A_{\varphi ,\psi }=G_{\exp \circ \phi ,\exp \circ \psi }.
\end{equation*}
\end{remark}

Applying Theorem \ref{thm:eq} we obtain the following

\begin{theorem}
Suppose that $G_{f,g}$ and $G_{\bar{f},\bar{g}}$ are \textit{generalized
weighted quasi-geometric means in }$I.$\textit{\ Then }%
\begin{equation*}
G_{\bar{f},\bar{g}}=G_{f,g}
\end{equation*}%
if, and only if, there exist $a\neq 0,b,c>0$ such that%
\begin{equation*}
\bar{f}(x)=b\left[ f\left( x\right) \right] ^{a},\qquad \bar{g}(x)=c\left[
g\left( x\right) \right] ^{a},\qquad x\in I.
\end{equation*}
\end{theorem}

\begin{proof}
Setting $\varphi =\log \circ f$, $\ \psi =\log \circ g$, \ $\Phi =\log \circ
\bar{f}$, $\Psi =\log \circ \bar{g}$, we see that\ $G_{f,g}=G_{\bar{f},\bar{g%
}}$ iff \ $A_{\Phi ,\Psi }=A_{\varphi ,\psi }$, and, in view of Theorem \ref%
{thm:eq}, iff there are $\alpha ,\beta ,\gamma \in \mathbb{R}$, $\alpha \neq
0,$ such that%
\begin{equation*}
\Phi (x)=\alpha \varphi (x)+\beta ,\qquad \Psi (x)=\alpha \psi (x)+\gamma
,\qquad x\in I,
\end{equation*}%
that is, iff%
\begin{equation*}
\log \circ \bar{f}(x)=\alpha \log \circ f(x)+\beta ,\qquad \log \circ \bar{g}%
(x)=\alpha \log \circ g(x)+\gamma ,\qquad x\in I,
\end{equation*}%
\begin{equation*}
\bar{f}(x)=e^{\beta }\left[ f(x)\right] ^{\alpha },\qquad \bar{g}%
(x)=e^{\gamma }\left[ g(x)\right] ^{\alpha },\qquad x\in I.
\end{equation*}%
Setting $a=\alpha $,$b=e^{\beta }$, $c=e^{\gamma }$, we obtain the result.
\end{proof}

\bigskip

\section{Operation $C_{f,g}$ and means}

Let $I\subset \left( 0,\infty \right)$ be an interval that is closed with
respect to the multiplication and let $f,g:I\rightarrow I$ be bijective
functions. Then, the two-variable function $C_{f,g}:I^{2}\rightarrow I$
given by%
\begin{equation}
C_{f,g}\left( x,y\right) :=\left( f\circ g\right) ^{-1}\left( f\left(
x\right) g\left( y\right) \right) \text{, \ \ \ \ }x,y\in I\text{,}  \tag{2}
\end{equation}%
is correctly defined.

In this section we examine when $C_{f,g}$ is a mean.

\begin{lemma}
If $C_{f,g}:I^{2}\rightarrow I$ defined by (2) is symmetric then $g=cf$ for
some real constant $c\in I$, $c\neq 0$; if moreover $C_{f,g}$ is a mean, then%
\begin{equation*}
C_{f,g}\left( x,y\right) =\mathcal{G}\left( x,y\right) \text{, \ \ \ \ }%
x,y\in I\,\text{, \ \ \ \ \ \ }
\end{equation*}%
where $\mathcal{G}\left( x,y\right) :=\sqrt{xy}$ is the geometric mean.
\end{lemma}

\begin{proof}
By (2), the symmetry of $C_{f,g}$ implies that $f\left( x\right) g\left(
y\right) =f\left( y\right) g\left( x\right) $ for all $x,y\in I$, that is $%
\frac{g\left( x\right) }{f\left( x\right) }=\frac{g\left( y\right) }{f\left(
y\right) }$ \ for all\ $x,y\in I,$ whence $\frac{g}{f}=c$ \ for some
positive real constant $c$, which proves the first result. If $C_{f,g}$ is a
mean, making use of (2) and the reflexivity of $C_{f,g}$, we get\ $f\left(
cf\left( x\right) \right) =f\left( x\right) \left( cf\left( x\right) \right)
$ for all $x\in I$ or, equivalently, \ $cf\left( cf\left( x\right) \right)
=\left( cf\left( x\right) \right) \left( cf\left( x\right) \right) $ for all
$x\in I,$ whence, in view of the equality $g=cf,$ we get $g\left( g\left(
x\right) \right) =\left( g\left( x\right) \right) ^{2}$ for all $x\in I$.
Since $g$ is bijective, it follows that $g\left( x\right) =x^{2}$ for all $%
x\in I.$ Now it is easy to get that $f\left( x\right) =x^{2}$ for all $x\in
I $. Consequently, making use of of (2), we obtain, for all $x,y\in I$,%
\begin{equation*}
C_{f,g}\left( x,y\right) =\sqrt{\sqrt{x^{2}y^{2}}}=\sqrt{xy}=\mathcal{G}%
\left( x,y\right) ,
\end{equation*}%
which completes the proof.
\end{proof}

Since every mean is reflexive (Remark 2(i)), we first consider conditions
for reflexivity of $C_{f,g}$. We begin with

\begin{lemma}
Let $I$ be a nontrivial interval that is closed with respect to the
multiplication. Assume that $f,g:I\rightarrow I $ are bijective functions
such that $C_{f,g}$ defined by (2) is reflexive, i.e. that
\begin{equation}
\left( f\circ g\right) ^{-1}\left( f\left( x\right)g\left( x\right) \right)
=x\text{, \ \ \ \ \ }x\in I.  \tag{3}
\end{equation}%
If $g$ has a fixpoint $x_{0}\in I,$ then $x_{0}=1;$ in particular $1$ must
belong to $I;$

if moreover $g$ is continuous and $I\subset \left[ 1,\infty \right) ,$ then $%
I=\left[ 1,\infty \right) $; $g$ is strictly increasing and either
\begin{equation*}
\text{\ \ \ }1<g\left( x\right) <x\text{, \ \ \ }x\in \left( 1,\infty
\right) \text{, }
\end{equation*}%
\ \ \ \ \ or%
\begin{equation*}
g\left( x\right) >x,\ \ \ \ \ \ x\in \left( 1,\infty \right) .
\end{equation*}%
\
\end{lemma}

\begin{proof}
From (3) we have
\begin{equation*}
f\left( g\left( x\right) \right) =f\left( x\right)g\left( x\right) \text{, \
\ \ }x\in I.
\end{equation*}%
Thus, if $g\left( x_{0}\right) =x_{0}$ for some $x_{0}\in I$, then $f\left(
x_{0}\right) =f\left( x_{0}\right)x_{0}$, so $x_{0}=1.$ Hence, if $I\subset %
\left[1,\infty \right) $ then, as $I$ is nontrivial and closed with respect
to the multiplication, it must be of the form $\left[1,\infty \right) $. Since $%
g $ has no fixpoints in $\left(1,\infty \right) $, the continuity of $g$
implies it must be increasing and either \ $1<g\left( x\right) <x$ \ for all
$x\in I$, or $g\left( x\right) >x$ for all $x\in I$.
\end{proof}

This lemma justifies the assumption that $I=\left(1,\infty \right) $ in our
considerations of reflexivity of $C_{f,g}$.

\begin{proposition}
Let $g:\left(1,\infty \right) \rightarrow \left(1,\infty \right) $ be
injective continuous and such that%
\begin{equation}
1<g\left( x\right) <x\text{, \ \ \ \ \ \ }x>1.  \tag{4}
\end{equation}%
Then there is no continuous function$\ f:\left(1,\infty \right) \rightarrow
\left(1,\infty \right) $ satisfying equation
\begin{equation}
f\left( g\left( x\right) \right) =f\left( x\right)g\left( x\right) \text{, \
\ \ \ \ }x\in \left( 1,\infty \right) ;  \tag{5}
\end{equation}%
in particular, there is no injective continuous function $f:\left( 1,\infty
\right) \rightarrow \left( 1,\infty \right) $ such that $C_{f,g}$ is
reflexive in $(1,\infty )$.
\end{proposition}

\begin{proof}
The continuity of $g$ and condition (4) imply that%
\begin{equation*}
\lim_{n\rightarrow \infty }g^{n}\left( x\right) =1\text{, \ \ \ \ \ }x>1,
\end{equation*}%
where $g^{n}$ denotes the $n$th iterate of $g$.

Assume that there is a continuous and strictly increasing function $f:\left(
1,\infty \right) \rightarrow \left( 1,\infty \right) $ satisfying (5).\ From
(5), by induction we get%
\begin{equation*}
f\left( g^{n}\left( x\right) \right) =f\left( x\right)
\prod_{k=1}^{n}g^{k}\left( x\right) \text{, \ \ \ \ \ \ \ }x\in \left(
1,\infty \right) \text{, \ }n\in \mathbb{N}\text{.}
\end{equation*}%
Since $f$ is nonnegative and increasing, it has a finite right-hand side
limit at $1$, denoted by $f\left( 1+\right) $. Letting here $n\rightarrow
\infty $, we obtain%
\begin{equation*}
f\left( 1+\right) =f\left( x\right) \prod_{k=1}^{\infty }g^{k}\left(
x\right) \text{, \ \ \ \ \ \ \ }x\in \left( 1,\infty \right) ,
\end{equation*}%
that is a contradiction, as the left side is real constant and the right
side is either strictly increasing or $\infty $.
\end{proof}

\begin{proposition}
Let $g:\left( 1,\infty \right) \rightarrow \left( 1,\infty \right) $ be
bijective, continuous and such that%
\begin{equation}
g\left( x\right) >x\text{, \ \ \ \ \ \ }x>1.  \tag{6}
\end{equation}%
Then the following conditions are equivalent:

(i) there is a continuous function $f:\left( 1,\infty \right) \rightarrow
\left( 1,\infty \right) $ such that $C_{f,g}$ is reflexive; \

(ii) there is a continuous function$\ f:\left( 1,\infty \right) \rightarrow
\left( 1,\infty \right) $ satisfying (5):
\begin{equation*}
f\left( g\left( x\right) \right) =f\left( x\right) g\left( x\right) ,\ \ \ \
\ x\in \left( 1,\infty \right) ;
\end{equation*}

(iii) there are a function $f : \left( 1,\infty \right) \rightarrow
\left( 1,\infty \right)$ and $c\geq 1$ such that%
\begin{equation}
f\left( x\right) =c\prod_{k=0}^{\infty }g^{-k}\left( x\right) \,,\text{ \ \
\ }x\in \left( 1,\infty \right) ,  \tag{7}
\end{equation}%
where $g^{-k}$ denotes the $k$th iterate of the function $g^{-1}$.
\end{proposition}

\begin{proof}
The implication $(i)\Longrightarrow (ii)$ is obvious.

Assume (ii)$.$ The assumptions on $g$ imply that $g^{-1},$ the inverse of $%
g, $ is continuous, strictly increasing, and, in view of (6),
\begin{equation*}
1<g^{-1}\left( x\right) <x\text{, \ \ \ \ \ \ }x>1.
\end{equation*}%
Consequently,
\begin{equation*}
\lim_{n\rightarrow \infty }g^{-n}\left( x\right) =1,\text{ \ \ \ \ \ \ }%
x>1\,,
\end{equation*}%
where $g^{-n}$ stands for the $n$th iterate of $g^{-1}\,.$

Assume that there is a continuous function $f$ such that equality (5) holds.
Replacing $x$ by $g^{-1}\left( x\right)$ in (5) we get%
\begin{equation*}
f\left( g^{-1}\left( x\right) \right) =\frac{f\left( x\right)}{x}\text{, \ \
\ \ \ }x\in \left( 1,\infty \right) ,
\end{equation*}%
whence, by induction,
\begin{equation*}
f\left( g^{-n}\left( x\right) \right) =\frac{f\left( x\right)} {%
\prod_{k=0}^{n-1}g^{-k}\left( x\right)} \text{, \ \ \ \ \ \ \ }x\in \left(
1,\infty \right) ,\text{ }n\in \mathbb{N}\text{.}
\end{equation*}%
Since $c:=f\left( 1+\right) \geq 1$ exists, letting here $n\rightarrow
\infty ,$ we obtain
\begin{equation*}
c=\frac{f\left( x\right)}{\prod_{k=0}^{\infty }g^{-k}\left( x\right)} \,,%
\text{ \ \ \ }x\in \left( 1,\infty \right) ,
\end{equation*}%
whence%
\begin{equation*}
f\left( x\right) =c\prod_{k=0}^{\infty }g^{-k}\left( x\right) \,,\text{ \ \
\ }x\in \left( 1,\infty \right) ,
\end{equation*}%
where the infinite product $\prod_{k=0}^{\infty }g^{-k}$ converges, and its
product is a continuous function, which proves that (iii) holds.

Assume (iii). If $f:\left( 1,\infty \right) \rightarrow \left( 1,\infty
\right) ~$is of the form (7) then, for every constant $c\geq 1$ and $x\in
\left( 1,\infty \right) $\ we have%
\begin{eqnarray*}
f\left( g\left( x\right) \right) &=&c\prod_{k=0}^{\infty }g^{-k}\left(
g\left( x\right) \right) =cg\left( x\right)\prod_{k=1}^{\infty }g^{-k}\left(
g\left( x\right) \right) \\
&=&\left( c\prod_{k=1}^{\infty }g^{-k}\left( g^{1}\left( x\right) \right)
\right)g\left( x\right) =\left( c\prod_{k=0}^{\infty }g^{-k}\left( x\right)
\right)g\left( x\right) \\
&=&f\left( x\right)g\left( x\right) .
\end{eqnarray*}%
which implies that $C_{f,g}$ is reflexive, so (i) holds.
\end{proof}

Now we prove the main result of this section.

\begin{theorem}
Let $f,g:\left( 1,\infty \right) \rightarrow \left( 1,\infty \right) $ be
continuous, strictly increasing and onto. The following conditions are
equivalent:

(i) \ \ the function $C_{f,g}$ is a bivariable strict mean in $\left(
1,\infty \right) ;$

(ii) the function $C_{f,g}$ is reflexive; \ \

(iii) the infinite product $\prod_{k=0}^{\infty }g^{-k}$ of iterates of $%
g^{-1}\,\ $converges uniformly on compact subsets of $\left( 1,\infty
\right) $ and,
\begin{equation}
f=\prod_{k=0}^{\infty }g^{-k};  \tag{8}
\end{equation}%
moreover $C_{f,g}=\mathcal{C}_{g},$ where $\mathcal{C}_{g}:\left( 1,\infty
\right) ^{2}\rightarrow \left( 1,\infty \right) $ defined by
\begin{equation}
\mathcal{C}_{g}\left( x,y\right) =\left( \prod_{k=0}^{\infty
}g^{-k+1}\right) ^{-1}\left( \prod_{k=0}^{\infty }g^{-k}\left(
x\right)g\left( y\right) \right) \text{, \ \ \ \ }x,y>1,  \tag{9}
\end{equation}%
is a strictly increasing bivariable mean in $\left( 1,\infty \right) $.
\end{theorem}

\begin{proof}
It is obvious that (i) implies (ii). Assume that (ii) holds. Then, applying
Propositions 1 and 2 and taking into account that $c:=f\left( 1+\right) =1$
(by the bijectivity of $f$) we get the first part of (iii); in particular
the uniform convergence on compact subsets of the infinite product of
iterates follows from (7) and the Dini theorem. The "moreover" result
follows immediately from (8) and the definition of $C_{f,g}$. Thus (ii)
implies (iii).

To prove (i) assuming (iii), note that the function $\mathcal{C}_{g}$ defined by (9) is
reflexive and strictly increasing in each variable so, by Remark 2(ii), it
is a strict mean.
\end{proof}

From the above theorem we obtain the following

\begin{corollary}
Let a continuous strictly increasing function $r:\left( 1,\infty \right)
\rightarrow \left( 1,\infty \right) $ be such that%
\begin{equation*}
1<r\left( x\right) <x\text{, \ \ \ \ }x>1,
\end{equation*}%
and the infinite product
\begin{equation*}
\prod_{k=0}^{\infty }r^{k}
\end{equation*}%
converges to a finite continuous function. Then the function $\mathcal{C}%
_{r}:\left( 1,\infty \right) ^{2}\rightarrow \left( 1,\infty \right) $
defined by
\begin{equation*}
\mathcal{C}_{r}\left( x,y\right) :=\left( \prod_{k=0}^{\infty
}r^{k-1}\right) ^{-1}\left( \prod_{k=0}^{\infty }r^{k}\left(
x\right)r^{-1}\left( y\right) \right) \text{, \ \ \ \ }x,y>1
\end{equation*}%
is a bivariable strict mean in $\left( 1,\infty \right) $.
\end{corollary}

\begin{proof}
It is enough to apply the previous result with $g:=r^{-1}$, \ $%
f:=\prod_{k=0}^{\infty }r^{k}$ and observe that $\mathcal{C}_{r}=C_{f,g}$.
\end{proof}

It seems to be interesting that the mean $\mathcal{C}_{r}$ is generated with
the aid of the iterates of the function $r$ satisfying the assumptions of
Corollary 1.

\bigskip The mean $\mathcal{C}_{r}$ is constructed with the aid of the
iterates of $r$, for convenience, we introduce the following

\begin{definition}
If $r:\left( 1,\infty \right) \rightarrow \left( 1,\infty \right) $
satisfies the assumptions of Corollary 1, then the function $\mathcal{C}_{r}$
is called a product iterative mean of generator $r$.
\end{definition}

\begin{example}
Applying the definition of $\mathcal{C}_{r}$ for the generator $r:\left(
1,\infty \right) \rightarrow \left( 1,\infty \right) $ given by $r\left(
x\right) =x^{w}$, where $w\in \left( 0,1\right) $ is fixed, we get weighted
geometric mean $\mathcal{C}_{r}\left( x,y\right) =$ $x^{w}y^{1-w}$ \ for all
$x,y\in \left( 1,\infty \right) .$ \ \ \
\end{example}

%\begin{example}
%Let $p\in \left( 0,1\right] $. The function $r:\left( 0,\infty \right)
%\rightarrow \left( 0,\infty \right) $ given by
%\begin{equation*}
%r\left( x\right) =\frac{px^{2}}{x+1}\text{, \ \ \ \ }x>0,
%\end{equation*}%
%is strictly increasing, and we have
%\begin{equation*}
%0<r\left( x\right) <px\text{, \ \ \ \ }x>0.
%\end{equation*}%
%It follows that in the case when $p<1$, the series $\sum_{k=0}^{\infty
%}r^{k} $ converges uniformly on compact sets, and, consequently, the
%function $r$ generates the iterative mean $\mathcal{D}_{r}$.\
%\end{example}

\begin{remark}
In Theorem 2 we assume that $I=\left( 1,\infty \right) .$ One could also get
the suitable results if $I$ is one the following intervals $\left[ 1,\infty
\right), \left( 0,1\right), \left( 0,1\right]$, and $\left( 0,\infty \right)$%
.
\end{remark}

\bigskip

\section{Invariant operation with respect to iterative mean-type mapping
need not be a mean}

In this section we focus our attention on the question whether the invariant
function $C_{f\circ g,g\circ h}$ occurring in the following counterpart of
(1):%
\begin{equation*}
C_{f\circ g,g\circ h}\circ \left( C_{f,g},C_{g,h}\right) =C_{f\circ g,g\circ
h},
\end{equation*}%
is a mean, if the coordinates of the mapping $\left( C_{f,g},C_{g,h}\right) $
defined by (2) are means.

If $C_{f,g}=\mathcal{C}_{g}$ and $C_{g,h}=\mathcal{C}_{h},$ then, in view of \cite[Theorem 1]{Matko18}, the function $C_{f\circ g,g\circ h}$\ is invariant with respect
to the mean type mapping $\left( \mathcal{C}_{g},\mathcal{C}_{h}\right) $.

On the other hand $\mathcal{C}_{g}$ (and $\mathcal{C}_{h}$) is a very
special case of generalized weighted quasi-geometric mean $G_{\varphi ,\psi
}:I^{2}\rightarrow I$,%
\begin{equation*}
G_{\varphi ,\psi }\left( x,y\right) =\left( \varphi \psi \right) ^{-1}\left(
\varphi \left( x\right) \psi \left( y\right) \right) \text{, \ \ \ \ }x,y\in
I\text{,}
\end{equation*}%
where $\varphi ,\psi :I\rightarrow (0,\infty )$. Indeed, with $\varphi
=\prod_{k=0}^{\infty }g^{-k}$ and $\psi =g$ we have $G_{\varphi ,\psi }=%
\mathcal{C}_{g}$ (and with $\psi =\prod_{k=0}^{\infty }h^{-k}$ and $\gamma
=h $ we have $G_{\psi ,\gamma }=\mathcal{C}_{h}$). Moreover, according to
\cite{Matko13} and Remark \ref{rem:bij}, if $G_{\varphi, \psi}$ and $G_{\psi
,\gamma }$ are two generalized weighted quasi-geometric means, then $%
G_{\varphi \psi ,\psi \gamma },$ the mean of the same type, is a unique mean
that is invariant with respect to the mean-type mapping $\left(
G_{\varphi,\psi },G_{\psi ,\gamma }\right) :I^{2}\rightarrow I^{2}$.

In this context the question arises whether the invariant function $%
C_{f\circ g,g\circ h}$ coincides with the invariant mean $G_{\varphi \psi
,\psi \gamma }?$ In view of Remark \ref{rem:bij} and (\cite{Matko13}), the
answer is positive, if $C_{f\circ g,g\circ h}$ is a mean, i.e. if there is a
suitable bijective function $u:\left( 1,\infty \right) \rightarrow \left(
1,\infty \right) $ such that $C_{f\circ g,g\circ h}=\mathcal{C}_{u}$.

We prove

\begin{theorem}
Assume that $f,g,h : (1,\infty) \to (1,\infty)$ are continuous, strictly increasing, and onto. Then $C_{f,g}$, $C_{g,h}$, and
$C_{f\circ g,g\circ h}$\ are means in $\left( 1,\infty \right) , $ iff
\begin{equation}
g=\prod_{i=0}^{\infty }h^{-i}\text{, \ \ \ \ \ \ \ \ \ }f=\prod_{j=0}^{%
\infty }\left( \prod_{i=0}^{\infty }h^{-i}\right) ^{-j}\text{, }  \tag{10}
\end{equation}%
and $h$ satisfies the composite functional equation
\begin{equation}
\text{\ } \prod_{j=0}^{\infty }\left( \prod_{i=0}^{\infty }h^{-i}\right)
^{-j+1} =\prod_{k=0}^{\infty }\left( \prod_{i=0}^{\infty }h^{-i+1}\right)
^{-k},  \tag{11}
\end{equation}%
where $i,j,k$ stand for the indices of iterates of the suitable functions.
\end{theorem}

\begin{proof}
Assume that $C_{f,g}$, $C_{g,h}$ and $C_{f\circ g,g\circ h}$\ are means in $%
\left( 1,\infty \right) .$ In view of Theorem 3, we have $C_{g,h}=\mathcal{C}%
_{h}$, $C_{f,g}=\mathcal{C}_{g},$ $C_{f\circ g,g\circ h}=\mathcal{C}_{g\circ
h}$, and
\begin{equation}
g=\prod_{i=0}^{\infty }h^{-i}\text{, \ \ \ \ \ \ }f=\prod_{j=0}^{\infty
}g^{-j}\text{, \ \ \ \ \ \ \ }f\circ g=\prod_{k=0}^{\infty }\left( g\circ
h\right) ^{-k}.  \tag{12}
\end{equation}%
It follows that (10) holds true, and the third of equalities (12) implies
that $h$ satisfies the composite functional equation,
\begin{equation*}
\text{\ }\left( \prod_{j=0}^{\infty }\left( \prod_{i=0}^{\infty
}h^{-i}\right) ^{-j}\right) \circ \left( \prod_{i=0}^{\infty }h^{-i}\right)
=\prod_{k=0}^{\infty }\left( \left( \prod_{i=0}^{\infty }h^{-i}\right) \circ
h\right) ^{-k},
\end{equation*}%
which simplifies to (11). The converse implication is obvious.
\end{proof}

Thus our question leads to rather complicate composite functional equation
(11). We pose the following

\begin{problem}
Determine strictly increasing bijective functions $h:\left( 1,\infty \right)
\rightarrow \left( 1,\infty \right) $ satisfying equation (11).
\end{problem}

Let us consider the following

\begin{example}
Take $w\in \left( 0,1\right) ,$ and define $h:\left( 1,\infty \right)
\rightarrow \left( 1,\infty \right) $ by $h\left( x\right) =x^{\frac{1}{w}}$%
. From Theorem 3 with $g$ and $f$ replaced, respectively, by $h$ and $g,$ we
get%
\begin{equation*}
g\left( x\right) =\prod_{k=0}^{\infty }h^{-k}\left( x\right)
=\prod_{k=0}^{\infty }x^{w^{k}}=x^{\frac{1}{1-w}}\text{, \ \ \ \ \ }x\in
\left( 1,\infty \right) ,
\end{equation*}%
and, for all $x,y>1$,
\begin{eqnarray*}
\mathcal{C}_{h}\left( x,y\right)  &=&\left( \prod_{k=0}^{\infty
}h^{-k+1}\right) ^{-1}\left( \prod_{k=0}^{\infty }h^{-k}\left( x\right)
h\left( y\right) \right)  \\
&=&\left( x^{\frac{1}{1-w}}y^{\frac{1}{w}}\right) ^{w(1-w)}=x^{w}y^{1-w}.
\end{eqnarray*}%
Similarly, as $g^{-1}\left( x\right) =x^{1-w},$ we get
\begin{equation*}
f\left( x\right) =\prod_{k=0}^{\infty }g^{-k}\left( x\right)
=\prod_{k=0}^{\infty }x^{\left( 1-w\right) ^{k}}=x^{\frac{1}{w}}\text{, \ \
\ \ \ }x\in \left( 1,\infty \right) ,
\end{equation*}%
and, for all $x,y>1,$%
\begin{eqnarray*}
\mathcal{C}_{g}\left( x,y\right)  &=&\left( \prod_{k=0}^{\infty
}g^{-k+1}\right) ^{-1}\left( \prod_{k=0}^{\infty }g^{-k}\left( x\right)
g\left( y\right) \right)  \\
&=&\left( x^{\frac{1}{w}}y^{\frac{1}{1-w}}\right) ^{w(1-w)}=x^{1-w}y^{w}.
\end{eqnarray*}%
Moreover, since $\left( f\circ g\right) \left( x\right) =x^{\frac{1}{w\left(
1-w\right) }}$, \ $\left( g\circ h\right) \left( x\right) =x^{\frac{1}{%
\left( 1-w\right) w}}$, \ we get, for all $x,y>1,$%
\begin{eqnarray*}
C_{f\circ g,g\circ h}\left( x,y\right)  &=&\left( \left( f\circ g\right)
\circ \left( g\circ h\right) \right) ^{-1}\left( \left( f\circ g\right)
\left( x\right) \cdot \left( g\circ h\right) \left( y\right) \right)  \\
&=&\left( x^{\frac{1}{w(1-w)}}y^{\frac{1}{(1-w)w}}\right)
^{w^{2}(1-w)^{2}}=\left( xy\right) ^{w(1-w)},
\end{eqnarray*}%
{\LARGE \ }

whence
\begin{equation*}
C_{f\circ g,g\circ h}=\mathcal{G}^{2w(1-w)},
\end{equation*}%
where $\mathcal{G}\left( x,y\right) =\sqrt{xy}$ ($x,y>1$) is the geometric
mean.

For every $w\in \left( 0,1\right) $ we have $\mathcal{G}\circ \left(
\mathcal{C}_{g},\mathcal{C}_{h}\right) =\mathcal{G}\,,$ and $\mathcal{G}$ is
a unique $\left( \mathcal{C}_{g},\mathcal{C}_{h}\right) $-invariant mean.
The function $C_{f\circ g,g\circ h}$ is also invariant for every $w\in
\left( 0,1\right) $, but $C_{f\circ g,g\circ h}$ is not a mean for any $w\in
\left( 0,1\right) $.
\end{example}

In the context of the above discussion let us consider the following

\begin{remark}
Under conditions of Theorem 4, the bijective functions $f,g,h:\left(
1,\infty \right) \rightarrow \left( 1,\infty \right) $ are increasing.
Consequently, they are almost everywhere differentiable and $f\left(
1+\right) =g\left( 1+\right) =h\left( 1+\right) =1$. Assume additionally
that $f\left( 1\right) =g\left( 1\right) =h\left( 1\right) =1$ and that $%
f,g,h$ are differentiable at the point $1$. If the functions $C_{f,g}$, $%
C_{g,h}$ and $C_{f\circ g,g\circ h}$\ were means in $\left[ 1,\infty \right)
$ then, by their reflexivity, we would have%
\begin{equation*}
f\left( g\left( x\right) \right) =f\left( x\right) g\left( x\right) \text{,
\ }g\left( h\left( x\right) \right) =g\left( x\right) h\left( x\right) \text{%
, \ }f\left( g^{2}\left( h\left( x\right) \right) \right) =f\left( g\left(
x\right) \right) g\left( h\left( x\right) \right) \text{, \ }
\end{equation*}%
for all $x\geq 1$, where $g^{2}$ is the second iterate of $g$.
Differentiating both sides of each of these equalities at $x=1$ and then setting%
\begin{equation*}
a:=f^{\prime }\left( 1\right) ,\text{ \ \ \ \ }b:=g^{\prime }\left( 1\right)
\text{, \ \ \ \ \ }c:=h^{\prime }\left( 1\right) ,
\end{equation*}%
we would get%
\begin{equation*}
ab=a+b\text{, \ \ \ \ }bc=b+c\text{, \ \ \ \ }ab^{2}c=ab+bc.
\end{equation*}%
Since this system has no solution satisfying the conditions $a\geq 1,b\geq
1,c\geq 1$, the functions $f,g,h$ do not exist.
\end{remark}

%{\LARGE This remark may incline us toward the conjecture that equation (13)
%has no an admissible solution}.

\section{\protect\bigskip Example of application of invariance identity}

In this section we show that knowing the invariant mean with respect to a
given continuous and strict mean-type mapping we can determine effectively
the limit of the sequence of its iterates, as well as, to find the form of
all continuous functions which are invariant with respect to this mapping.

\begin{theorem}
\label{thm:ex} Let $f,g:(0,\infty )\rightarrow \left( 0,\infty \right) $ be
continuous strictly increasing functions such that the function $\frac{f}{g}$
is strictly increasing. Then

(i) \ the functions $M,N:\left( 0,\infty \right) ^{2}\rightarrow \left(
0,\infty \right) $ given by%
\begin{equation*}
M\left( x,y\right) :=f^{-1}\left( \frac{g(y)}{g(x)}f\left( x\right) \right)
\text{, \ \ \ \ \ }N\left( x,y\right) :=f^{-1}\left( \frac{g(x)}{g(y)}%
f\left( y\right) \right)
\end{equation*}%
are generalized weighted quasi-geometric means, as $M=G_{\frac{f}{g},g}$ and
$N=G_{g,\frac{f}{g}};$

(ii) the quasi-geometric mean $G_{f}:\left( 0,\infty \right) ^{2}\rightarrow
\left( 0,\infty \right) $ given by%
\begin{equation*}
G_{f}\left( x,y\right) =f^{-1}\left( \sqrt{f\left( x\right)f\left( y\right)}
\right)
\end{equation*}%
is invariant with respect to the mean-type mapping $\left( M,N\right) ,$ i.e.%
\begin{equation*}
G_{f}\circ \left( M,N\right) =G_{f};
\end{equation*}

(iii) the sequence $\left( \left( M,N\right) ^{n}:n\in \mathbb{N}\right) $
of iterates of the mean-type mapping $\left( M,N\right) :\left( 0,\infty
\right) ^{2}\rightarrow \left( 0,\infty \right) ^{2}$ converges uniformly on
compact subsets of $\left( 0,\infty \right) ^{2}$ and%
\begin{equation*}
\lim_{n\rightarrow \infty }\left( M,N\right) ^{n}=\left( G_{f},G_{f}\right) ;
\end{equation*}

(iv) a continuous function $\Phi :\left( 0,\infty \right) ^{2}\rightarrow
\mathbb{R}$ is invariant with respect to the mean-type mapping $\left(
M,N\right) $, i.e. $\Phi $ satisfies the equation%
\begin{equation*}
\Phi \left( M\left( x,y\right) ,N\left( x,y\right) \right) =\Phi \left(
x,y\right) \text{, \ \ \ }x,y>0\text{,}
\end{equation*}%
if and only if there is a continuous single variable function $\varphi
:\left( 0,\infty \right) \rightarrow \mathbb{R}$ such that%
\begin{equation*}
\Phi \left( x,y\right) =\varphi \left( G_{f}\left( x,y\right) \right) \text{%
, \ \ \ }x,y>0.
\end{equation*}
\end{theorem}

\begin{proof}
(i) By assumption, $M$ and $N$ are both reflexive and strictly increasing
in each variable so, by Remark 2(ii), they are strict means. The
remaining result is easy to verify.

(ii) By the definitions of $G_{f},$ $M$ and $N$, we have, for all $x,y>0,$%
\begin{eqnarray*}
G_{f}\circ \left( M,N\right) \left( x,y\right) &=&G_{f}\left( \left( M\left(
x,y\right) ,N\left( x,y\right) \right) \right) \\
&=& f^{-1}\left(\sqrt{f\left[f^{-1}\left( \frac{g(y)}{g(x)}f\left( x\right)
\right) \right] \cdot f\left[f^{-1}\left( \frac{g(x)}{g(y)}f\left( y\right)
\right) \right]} \right) \\
&=&f^{-1}\left( \sqrt{f\left( x\right) f\left( y\right)} \right)
=G_{f}\left( x,y\right) ,
\end{eqnarray*}%
which proves (ii).

Since $M$ and $N$ are strict and continuous means, (iii) follows from (ii)
and the result of \cite{AnnMath.Sil.1999} (see also \cite{GrazerMathBer.209}).

(iv) Assume that $\Phi :\left( 0,\infty \right) ^{2}\rightarrow \mathbb{R}$
is invariant with respect to the mean-type mapping $\left( M,N\right) $, so%
\begin{equation*}
\Phi \left( x,y\right) =\Phi \left( M\left( x,y\right) ,N\left( x,y\right)
\right) =\left( \Phi \circ \left( M,N\right) \right) \left( x,y\right) \text{%
, \ \ \ }x,y>0.
\end{equation*}%
Hence, by induction,
\begin{equation*}
\Phi \left( x,y\right) =\left( \Phi \circ \left( M,N\right) ^{n}\left(
x,y\right) \right) \text{, \ \ \ \ }n\in \mathbb{N}\text{; }x,y>0.
\end{equation*}%
Hence, in view of (iii), by the continuity of $\Phi $, we have, for all $%
x,y>0$,
\begin{eqnarray*}
\Phi \left( x,y\right) &=&\lim_{n\rightarrow \infty }\left( \Phi \circ
\left( M,N\right) ^{n}\left( x,y\right) \right) =\left( \Phi \circ \left(
G_{f},G_{f}\right) \right) \left( x,y\right) \\
&=&\Phi \left( G_{f}\left( x,y\right) ,G_{f}\left( x,y\right) \right)
=\varphi \left( G_{f}\left( x,y\right) \right) \text{, }
\end{eqnarray*}%
where $\varphi :\left( 0,\infty \right) \rightarrow \mathbb{R}$ is defined
by
\begin{equation*}
\varphi \left( x\right) :=\Phi \left( x,x\right) \text{, \ \ \ \ }x>0\text{.
}
\end{equation*}%
To prove the converse implication, take an arbitrary function $\varphi
:\left( 0,\infty \right) \rightarrow \mathbb{R}$ and define $\Phi :\left(
0,\infty \right) ^{2}\rightarrow \mathbb{R}$ by%
\begin{equation*}
\Phi \left( x,y\right) :=\varphi \left( G_{f}\left( x,y\right) \right) .
\end{equation*}%
Then, for arbitrary $x,y>0,$ using the definition of $\Phi $ and (ii) \
we get%
\begin{equation*}
\left( \Phi \circ \left( M,N\right) \right) \left( x,y\right) =\varphi
\left( G_{f}\left( M\left( x,y\right) ,N\left( x,y\right) \right) \right)
=\varphi \left( G_{f}\left( x,y\right) \right) =\Phi \left( x,y\right) ,
\end{equation*}%
which proves (v).
\end{proof}

\begin{remark}
Let $f ,g, h :I\rightarrow (0,\infty)$ be continuous and
all strictly increasing or all strictly decreasing. Using a similar argument to the proof of Theorem \ref{thm:ex}, we can show that a continuous function $\Phi : I^2 \rightarrow \mathbb{R}$ is invariant with respect to the mean-type mapping $(G_{f,g},G_{g,h}) : I^2 \rightarrow I$ if and only if there is a continuous single variable function $\varphi
:I \rightarrow \mathbb{R}$ such that $\Phi \left( x,y\right) =\varphi \left( G_{fg,gh}\left( x,y\right) \right)$ for all $x,y\in I$.

\end{remark}

A similar argument to the proof of Theorem \ref{thm:ex}, shows the following

\begin{corollary}
Let $f,g:\mathbb{R}\rightarrow \mathbb{R} $ be continuous strictly
increasing functions such that the function $f-g$ is strictlay increasing. Then

(i) \ The functions $M,N:\mathbb{R} ^{2}\rightarrow \mathbb{R} $ given by%
\begin{equation*}
M\left( x,y\right) :=f^{-1}\left( g(y)-g(x)+f\left( x\right) \right) \text{,
\ \ \ \ \ }N\left( x,y\right) :=f^{-1}\left( g(x)-g(y)+f\left( y\right)
\right)
\end{equation*}%
are generalized weighted strict quasi-arithmetic means, as $M=A_{f-g,g}$ and
$N=A_{g,f-g};$

(ii) the quasi-arithmetic mean $A_{f}:\mathbb{R} ^{2}\rightarrow \mathbb{R} $
given by%
\begin{equation*}
A_{f}\left( x,y\right) =f^{-1}\left( \frac{f(x)+f(y)}{2} \right)
\end{equation*}%
is invariant with respect to the mean-type mapping $\left( M,N\right) ,$ i.e.%
\begin{equation*}
A_{f}\circ \left( M,N\right) =A_{f};
\end{equation*}

(iii) the sequence $\left( \left( M,N\right) ^{n}:n\in \mathbb{N}\right) $
of iterates of the mean-type mapping $\left( M,N\right) :\mathbb{R}%
^{2}\rightarrow \mathbb{R}^{2}$ converges uniformly on compact subsets of $%
\mathbb{R}^{2}$ and%
\begin{equation*}
\lim_{n\rightarrow \infty }\left( M,N\right) ^{n}=\left( A_{f},A_{f}\right) ;
\end{equation*}

(iv) a continuous function $\Phi :\mathbb{R} ^{2}\rightarrow \mathbb{R}$ is
invariant with respect to the mean-type mapping $\left( M,N\right) $, i.e. $%
\Phi $ satisfies the equation%
\begin{equation*}
\Phi \left( M\left( x,y\right) ,N\left( x,y\right) \right) =\Phi \left(
x,y\right) \text{, \ \ \ }x,y \in \mathbb{R}\text{,}
\end{equation*}%
if and only if there is a continuous single variable function $\varphi :%
\mathbb{R} \rightarrow \mathbb{R}$ such that%
\begin{equation*}
\Phi \left( x,y\right) =\varphi \left( A_{f}\left( x,y\right) \right) \text{%
, \ \ \ }x,y\in \mathbb{R}.
\end{equation*}
\end{corollary}

\section{\protect\bigskip Equality of two considered operations}

Let $I\subset (0,\infty)$ be an interval that is closed with respect to the addition and
multiplication. In this section, given two operations $C_{f,g}$ and $C_{\phi
,\psi }$ (resp.\ $D_{f,g}$ and $D_{\bar{f} ,\bar{g} }$), we determine the conditions under which they are equal.

\begin{proposition}
Let $f,g,\phi ,\psi \colon I\rightarrow I$ be bijective and continuous
functions. Then we have $C_{f,g}=C_{\phi ,\psi }$ iff $g=\psi $ and $f=a\phi
$, for some $a\in I$, $a\neq 0$.
\end{proposition}

\begin{proof}
(Necessity) Suppose that
\begin{equation*}
C_{f,g}(x,y)=C_{\phi ,\psi }(x,y)\quad x,y\in I,
\end{equation*}%
that is,
\begin{equation*}
\left( f\circ g\right) ^{-1}\left( f\left( x\right) g\left( y\right) \right)
=\left( \phi \circ \psi \right) ^{-1}\left( \phi \left( x\right) \psi \left(
y\right) \right) \quad x,y\in I.
\end{equation*}%
Setting $\alpha :=\left( f\circ g\right) \circ \left( \phi \circ \psi
\right) ^{-1}$, $\beta :=f\circ \phi ^{-1}$, $\gamma :=g\circ \psi ^{-1}$, $%
u=\phi (x)$, and $v=\psi (y)$ we get
\begin{equation*}
\alpha (uv)=\beta (u)\gamma (v)\quad u\in \phi \left( I\right) ,v\in \psi
\left( I\right) .
\end{equation*}%
The latter equation is Pexider's functional equation and has the following
solutions (see, e.g., \cite{Acz2006})
\begin{equation*}
\beta (u)=au^{c}\quad \mbox{and}\quad \gamma (v)=bv^{c}\quad u\in \phi
\left( I\right) ,v\in \psi \left( I\right) ,
\end{equation*}%
for some $a,b,c\in I$, $a,b,c\neq 0$. From these equations we get
\begin{equation*}
f(x)=a\phi (x)^{c}\quad \mbox{and}\quad g(y)=b\psi (y)^{c}\quad x,y\in I,
\end{equation*}%
for some $a,b,c\in I$, $a,b,c\neq 0$. Thus, on the one hand, we get
\begin{equation*}
\left( f\circ g\right) ^{-1}\left( f(x)g(y)\right) =\psi ^{-1}\left( \left(
\frac{1}{b}\phi ^{-1}\left( \left( \phi (x)^{c}b\psi (y)^{c}\right) ^{\frac{1%
}{c}}\right) \right) ^{\frac{1}{c}}\right) \quad x,y\in I,
\end{equation*}%
for some $b,c\in I$, $b,c\neq 0$, and on the other hand, we get
\begin{equation*}
\left( \phi \circ \psi \right) ^{-1}\left( \phi (x)\psi (y)\right) =\psi
^{-1}\left( \phi ^{-1}\left( \phi (x)\psi (y)\right) \right) \quad x,y\in I.
\end{equation*}%
Hence, we get
\begin{equation*}
\psi ^{-1}\left( \left( \frac{1}{b}\phi ^{-1}\left( \left( \phi (x)^{c}b\psi
(y)^{c}\right) ^{\frac{1}{c}}\right) \right) ^{\frac{1}{c}}\right) =\psi
^{-1}\left( \phi ^{-1}\left( \phi (x)\psi (y)\right) \right) \quad x,y\in I,
\end{equation*}%
which implies that $b=c=1$. Thus, for all $x\in I$, we get $g(x)=\psi (x)$
and $f(x)=a\phi (x)$, for some $a\in I$, $a\neq 0$.

(Sufficiency) Obvious.
\end{proof}

\begin{proposition}
Let $f,g,\bar{f},\bar{g}\colon I\rightarrow I$ be bijective and continuous
functions. Then we have $D_{f,g}=D_{\bar{f},\bar{g}}$ iff $f=\bar{f}+a$, for
some $a\in I$, and $g=\bar{g}$.
\end{proposition}

\begin{proof}
(Necessity) Suppose that
\begin{equation*}
D_{f,g}(x,y)=D_{\bar{f},\bar{g}}(x,y),\quad x,y\in I,
\end{equation*}%
that is,
\begin{equation*}
\left( f\circ g\right) ^{-1}\left( f\left( x\right) +g\left( y\right)
\right) =\left( \bar{f}\circ \bar{g}\right) ^{-1}\left( \bar{f}\left(
x\right) +\bar{g}\left( y\right) \right) ,\quad x,y\in I.
\end{equation*}%
Setting $\alpha :=\left( f\circ g\right) \circ \left( \bar{f}\circ \bar{g}%
\right) ^{-1}$, $\beta :=f\circ \left( \bar{f}\right) ^{-1}$, $\gamma
:=g\circ \left( \bar{g}\right) ^{-1}$, $u=\bar{f}(x)$, and $v=\bar{g}(y)$ we
get
\begin{equation*}
\alpha (u+v)=\beta (u)+\gamma (v),\quad u\in \bar{f}\left( I\right) ,v\in
\bar{g}\left( I\right) .
\end{equation*}%
The latter equation is Pexider's functional equation and has the following
solutions (see, e.g., \cite{Acz2006})
\begin{equation*}
\beta (u)=au+b\quad \mbox{and}\quad \gamma (v)=av+c,\quad u\in \bar{f}\left(
I\right) ,v\in \bar{g}\left( I\right) ,
\end{equation*}%
for some $a,b,c\in I$, $a\neq 0$. From these equations we get
\begin{equation*}
f(x)=a\bar{f}(x)+b\quad \mbox{and}\quad g(y)=a\bar{g}(y)+c,\quad x,y\in I,
\end{equation*}%
for some $a,b,c\in I$, $a\neq 0$. Thus, on the one hand, we get
\begin{equation*}
\left( f\circ g\right) ^{-1}\left( f(x)+g(y)\right) =\left( \bar{g}\right)
^{-1}\left( \frac{\left( \bar{f}\right) ^{-1}\left( \bar{f}(x)+\bar{g}(y)+%
\frac{c}{a}\right) -c}{a}\right) ,\quad x,y\in I,
\end{equation*}%
for some $a,c\in I$, $a\neq 0$, and on the other hand, we get
\begin{equation*}
\left( \bar{f}\circ \bar{g}\right) ^{-1}\left( \bar{f}(x)+\bar{g}(y)\right)
=\left( \bar{g}\right) ^{-1}\left( \left( \bar{f}\right) ^{-1}\left( \bar{f}%
(x)+\bar{g}(y)\right) \right) ,\quad x,y\in I.
\end{equation*}%
Hence, we get
\begin{equation*}
\left( \bar{g}\right) ^{-1}\left( \frac{\left( \bar{f}\right) ^{-1}\left(
\bar{f}(x)+\bar{g}(y)+\frac{c}{a}\right) -c}{a}\right) =\left( \bar{g}%
\right) ^{-1}\left( \left( \bar{f}\right) ^{-1}\left( \bar{f}(x)+\bar{g}%
(y)\right) \right) ,\quad x,y\in I,
\end{equation*}%
which implies that $a=1$ and $c=0$. Thus, for all $x\in I$, we get $f(x)=%
\bar{f}(x)+b$, for some $b\in I$, and $g(x)=\bar{g}(x)$.

(Sufficiency) Obvious.
\end{proof}

\bigskip

%\section{\protect\bigskip Final remarks}

%\begin{remark}
%Considering the operations $C_{f,g}$ and $D_{f,g}$, we can define another
%operation
%\begin{equation*}
%Q_{f,g}\left( x,y\right) =\frac{C_{f,g}(x,y)}{D_{f,g}(x,y)}\text{, \ \ \ \ }%
%x,y\in (1,\infty ).
%\end{equation*}%
%We observe that
%\begin{equation*}
%Q_{2id,2id}=H(x,y)\text{, \ \ \ \ }x,y\in (1,\infty ),
%\end{equation*}%
%where $H(x,y):=\frac{2xy}{x+y}$ is the harmonic mean. It is known that $%
%G\circ (A,H)=G,$ where $A(x,y):=\frac{x+y}{2}$ is the arithmetic mean. We
%can consider the problem to find an invariance identity for $C_{f,g},D_{f,g},
%$ and $Q_{f,g}$ and the conditions under which $Q_{f,g}$ is a mean.
%\end{remark}
%
%\begin{remark}
%Considering the operations $A_{f,g}$ and $G_{f,g}$, we can define another
%operation
%\begin{equation*}
%B_{f,g}\left( x,y\right) =\frac{G_{f,g}(x,y)^{2}}{A_{f,g}(x,y)}\text{, \ \ \
%\ }x,y\in (1,\infty ).
%\end{equation*}%
%We observe that
%\begin{equation*}
%B_{id,id}=H(x,y)\text{, \ \ \ \ }x,y\in (1,\infty ),
%\end{equation*}%
%where $H(x,y):=\frac{2xy}{x+y}$ is the harmonic mean. It is known that $%
%G\circ (A,H)=G,$ where $A(x,y):=\frac{x+y}{2}$ is the arithmetic mean. We
%can consider the problem to find an invariance identity for $M_{f,g},A_{f,g},
%$ and $B_{f,g}$ and the conditions under which $B_{f,g}$ is a mean.
%\end{remark}
%
%{\LARGE Since these two remarks suggest some further considerations, they
%should be omitted here.}

\section*{Acknowledgements}

This research is partly supported by the internal research project
R-AGR-0500 of the University of Luxembourg and by the Luxembourg National
Research Fund R-AGR-3080. %\bigskip
%
%One could either mention it here or try to write another paper.

\end{document}